\documentclass[11pt]{article}

\usepackage[letterpaper,centering,margin=1.5in]{geometry}

\usepackage{mathptmx}

\usepackage[shortlabels]{enumitem}

\usepackage{datetime}

\usepackage{eucal}

\usepackage[all]{xy}

\usepackage{amssymb, amsmath, amsthm}

\numberwithin{equation}{section}

\usepackage[pdftex, hyperindex=true]{hyperref}

\hypersetup{
  colorlinks = true,
  linkcolor = black,
  urlcolor = blue,
  citecolor = black,
}

\renewcommand{\phi}{\varphi}

\newcommand{\bfC}{\mathbf{C}}

\newcommand{\bfE}{\mathbf{E}}
\newcommand{\bfF}{\mathbf{F}}

\newcommand{\bfM}{\mathbf{M}}
\newcommand{\bfS}{\mathbf{S}}

\newcommand{\bbC}{\mathbb{C}}
\newcommand{\bbF}{\mathbb{F}}

\newcommand{\bbN}{\mathbb{N}}
\newcommand{\bbQ}{\mathbb{Q}}
\newcommand{\bbS}{\mathbb{S}}

\newcommand{\bbZ}{\mathbb{Z}}

\newcommand{\op}{\mathrm{op}}

\newcommand{\rmB}{\mathrm{B}}

\newcommand{\rmE}{\mathrm{E}}
\newcommand{\rmF}{\mathrm{F}}

\newcommand{\rmH}{\mathrm{H}}

\newcommand{\rmK}{\mathrm{K}}

\newcommand{\rmM}{\mathrm{M}}
\newcommand{\rmQ}{\mathrm{Q}}
\newcommand{\rmS}{\mathrm{S}}

\DeclareMathOperator{\id}{id}
\DeclareMathOperator{\Map}{Map}

\DeclareMathOperator{\End}{End}

\DeclareMathOperator{\Aut}{Aut}
\DeclareMathOperator{\Pre}{Pre}

\DeclareMathOperator{\GL}{GL}

\DeclareMathOperator{\colim}{colim}

\newcommand{\Sets}{\mathrm{Sets}}
\newcommand{\Mod}{\mathrm{Mod}}
\newcommand{\Abel}{\mathrm{Abel}}
\newcommand{\Groups}{\mathrm{Groups}}
\newcommand{\Monoids}{\mathrm{Monoids}}
\newcommand{\Nil}{\mathrm{Nil}}

\newcommand{\Boole}{\mathrm{Boole}}
\newcommand{\Cantor}{\mathrm{Cantor}}
\newcommand{\Post}{\mathrm{Post}}

\usepackage{graphicx}

\DeclareRobustCommand{\coprod}{\mathop{\text{\fakecoprod}}}
\newcommand{\fakecoprod}{%
  \sbox0{$\prod$}%
  \smash{\raisebox{\dimexpr.9625\depth-\dp0}{\scalebox{1}[-1]{$\prod$}}}%
  \vphantom{$\prod$}%
}

\newtheorem{theorem}{Theorem}[section]
\newtheorem{proposition}[theorem]{Proposition}
\newtheorem{corollary}[theorem]{Corollary}
\newtheorem{lemma}[theorem]{Lemma}

\newtheorem{theoremABC}{Theorem}

\theoremstyle{definition}

\newtheorem{definition}[theorem]{Definition}
\newtheorem{example}[theorem]{Example}
\newtheorem{remark}[theorem]{Remark}

\newtheorem{defn}[theorem]{Definition}

\title{\bf Boolean algebras, Morita invariance, and the algebraic~K-theory of Lawvere theories}

\author{Anna Marie Bohmann and Markus Szymik}

\newdateformat{mydate}{\monthname~\twodigit{\THEYEAR}}
\date{\mydate\today}

\begin{document}

\maketitle

\renewcommand{\abstractname}{\vspace{-2\baselineskip}}
\begin{abstract}
\noindent
The algebraic~K-theory of Lawvere theories is a conceptual device to elucidate the stable homology of the symmetry groups of algebraic structures such as the permutation groups and the automorphism groups of free groups. In this paper, we fully address the question of how Morita equivalence classes of Lawvere theories interact with algebraic~K-theory. On the one hand, we show that the higher algebraic~K-theory is invariant under passage to matrix theories. On the other hand, we show that the higher algebraic~K-theory is not fully Morita invariant because of the behavior of idempotents in non-additive contexts: We compute the~K-theory of all Lawvere theories Morita equivalent to the theory of Boolean algebras.
\end{abstract}


Quillen's seminal work~\cite{Quillen:Annals} used algebraic~K-theory to organize our thinking about the stable homology of general linear groups. This initiated generalizations to contexts far broader than that of rings. In this paper, we restrict our attention to Lawvere's algebraic theories. These structures provide a happy medium between rings and symmetric monoidal categories: no higher-categorical language is required, and they are much more flexible than rings. For instance, the stable homology of the symmetric groups and of the automorphism groups of free groups~\cite{Galatius} fit into this context as well. Our results are motivated by such stable homology computations, the starting point being the following fact~(see Theorem~\ref{thm:colimit}): For every Lawvere theory~$T$, there is an isomorphism
\[
\colim_r\rmH_*(\Aut(T_r))\cong\rmH_*(\Omega^\infty_0\rmK(T))
\]
between the stable homology of the automorphism groups of finitely generated free objects of the theory~$T$ and the homology of the zero component~$\Omega^\infty_0\rmK(T)$ of the algebraic~K-theory space~$\Omega^\infty\rmK(T)$. The surprising power of this observation comes from two sources. First, the~K-theory space or spectrum is often easier to describe than its homology. This happens, for instance for the symmetric groups. Second, algebraic~K-theory can sometimes be computed without explicitly using the groups~$\Aut(T_r)$~(see~\cite{Szymik+Wahl}, for example). Here, we present a new stable homology computation, for the theory of Boolean algebras, phrased once again in terms of algebraic~K-theory.

\begin{theoremABC}\label{thmABC:Boole}
For the algebraic~K-theory of the Lawvere theory~$\Boole$ of Boolean algebras, we have~$\rmK_*(\Boole)\cong\pi_*(\bbS)/2\text{\upshape--power torsion}$.
\end{theoremABC}

In this result, the groups~$\pi_*(\bbS)$ are the stable homotopy groups of spheres. These groups are the~K-groups of the initial Lawvere theory of sets, but the resulting homomorphism to~$\rmK_*(\Boole)$ is not surjective~(see Proposition~\ref{prop:Morava}). Theorem~\ref{thmABC:Boole} is a consequence of the following spectrum-level result, proved as Theorem~\ref{thm:K(Post_v)}, which is also a generalization from Boolean algebras to many-valued logics as modeled by the Lawvere theories~$\Post_v$ of Post algebras of valence~$v$. The superscript in $R^\times$ refers to the units of a ring~(spectrum)~$R$.

\begin{theoremABC}\label{thmABC:Post}
For every integer~$v\geqslant 2$, there is a homotopy pullback square
\[
\xymatrix{
\rmK(\Post_v)\ar[r]\ar[d] & \bbS[1/v]^\times\ar[d]\\
\rmH\bbZ\ar[r]_-v  & \rmH\bbZ[1/v]^\times
}
\]
of spectra.
\end{theoremABC}

These results can be conceptualized in terms of Morita invariance. Two rings are called Morita equivalent if they have equivalent categories of modules. Morita equivalent rings must have isomorphic higher algebraic~K-groups~(see~\cite[IV~Ex.~1.21, IV~6.3.5]{Weibel}). More generally, two Lawvere theories are called Morita equivalent if their categories of models are equivalent. This is the case if and only one of them is an idempotent modification of a matrix theory of the other; see the brief review in Section~\ref{sec:matrix}. We first prove a positive result~(see Theorem~\ref{thm:matrix_invariance}), which we expect to be a useful tool in stable homology computations.

\begin{theoremABC}
The higher algebraic~K-theory of Lawvere theories is invariant under passage to matrix theories.
\end{theoremABC}

Because we define the algebraic~K-theory of Lawvere theories in terms of free models, there is no hope of extending this result to~$\rmK_0$: there are even Morita equivalent rings that have non-isomorphic~$\rmK_0$'s when these~K-groups are defined using free modules only. This is due, of course, to the presence of projectives that are not free. Arguably, the ability to detect those non-free projectives is one desirable feature of lower~K-theory. For rings, we could have built that feature into our theory by completing idempotents.  In an additive category, all retracts have complements, and this completion does not change the higher algebraic~K-theory, only~$\rmK_0$. However, for general Lawvere theories, this fix for~$\rmK_0$ is not possible without changing the higher algebraic~K-theory: we show that completing at idempotents can change the higher algebraic~K-groups. In fact, since the Lawvere theories~$\Post_v$ are all Morita equivalent, our computations in Theorem~\ref{thmABC:Post} show the following:

\begin{theoremABC}\label{thmABC:syntax_semantics}
The higher algebraic~K-theory of Lawvere theories is not Morita invariant. 
\end{theoremABC}

We can rephrase this result in terms of the~``syntax'' of a Lawvere theory, which is defined by the{\it~free} models, and its~``semantics,'' which comprises {\it all} models: the higher algebraic~K-theory of an algebraic theory depends essentially on the syntax of the theory, rather than merely its semantics. We refer to Lawvere's writings~\cite{Lawvere:PNAS, Lawvere:Dialectica, Lawvere:Introduction} for the distinction between syntax and semantics in this context. From the perspective of mathematical logic and topos theory~\cite{Cartier}, different notions of equivalence of theories, both semantic and syntactical, have recently been discussed and compared in~\cite{Barrett--Halvorson, Tsementzis}.


While stable homology computations are one of our motivations for considering the algebraic K-theory of Lawvere theories, the related issue of homological stability is not the focus of the present work. We refer to the paper~\cite{R-WW} by Randal-Williams and Wahl, which discusses the homological stability problem in a more general framework than ours. Nonetheless, the specific setting of Lawvere theories balances rigidity and flexibility in a way that suggests it to be particularly amenable to homological stability questions as well. Additional motivation for the algebraic K-theory of Lawvere theories, in the form of multiplicative matters and applications to assembly maps, is discussed in~\cite{Bohmann+Szymik:assembly}. 

{\bf Outline.} Section~\ref{sec:theories} recalls Lawvere's categorical approach to universal algebra and sets up the notation that we use. In Section~\ref{sec:K}, we define the~K-theory of algebraic theories and show that it encodes the stable homology of the automorphism groups of the free models. A plethora of examples that do not come from rings and modules are presented in Section~\ref{sec:examples} before we start our discussion of Morita invariance with our theorem for matrix theories in Section~\ref{sec:matrix}. The final Section~\ref{sec:Boole} contains the computation for the theory of Boolean algebras and all theories equivalent to them.


\section{Lawvere theories}\label{sec:theories}

We need to review the basic notions and set up our notation for Lawvere theories~\cite{Lawvere:PNAS}. Some textbook references are~\cite{Pareigis, Schubert, Borceux, ARV}.

Choose a skeleton~$\bfE$ of the category of finite sets and~(all) maps between them. For each integer~\hbox{$r\geqslant0$} such a category has a unique object with precisely~$r$ elements, and there are no other objects. For the sake of explicitness, let us choose the model~\hbox{$\underline{r}=\{a\in\bbZ\,|\,1\leqslant a\leqslant r\}$} for such a set. A set with~$r+s$ elements is the~(categorical) sum, or co-product, of a set with~$r$ elements and a set with~$s$ elements.

\begin{definition}
A {\em Lawvere theory}~$T=(\bfF_T,\rmF_T)$ is a pair consisting of a small category~$\bfF_T$ together with a functor~\hbox{$\rmF_T\colon\bfE\to\bfF_T$} that is bijective on sets of objects and that preserves sums. This means that the canonical map~\hbox{$\rmF_T(\underline{r})+\rmF_T(\underline{s})\to\rmF_T(\underline{r}+\underline{s})$} induced by the canonical injections is an isomorphism for all sets~$\underline{r}$ and~$\underline{s}$ in~$\bfE$. 
\end{definition}

The image of the set~$\underline{r}$ with~$r$ elements under the functor~$\rmF_T\colon\bfE\to\bfF_T$ will be written~$T_r$, so that the object~$T_r$ is the sum in the category~$\bfF_T$ of~$r$ copies of the object~$T_1$. 

We recall two of the most important classes of examples of Lawvere theories.

\begin{example}\label{ex:Amodules}
Let~$A$ be a ring.  Let~$\bfF_A$ be the full subcategory of the category~$\Mod_A$ of~$A$--modules spanned by the modules~$A^{\oplus r}$ for~$r\geqslant0$. This category is a skeleton of the category of finitely generated, free~$A$--modules.
The functor~\hbox{$\rmF_A\colon\bfE\to\bfF_A$} that sends the set with~$r$ elements to the free module~$A^{\oplus r}$ with~$r$ generators is a Lawvere theory, called the theory of~$A$--modules. Note that~$A^{\oplus 0}=0$ is the~$0$ module. In particular, for the initial ring~$A=\bbZ$, we have the Lawvere theory of abelian groups. 
\end{example}

Rings can be very complicated, and this is even more true for Lawvere theories, which are significantly more general.

\begin{example}\label{ex:Gsets}
Let~$G$ be a group.  Let~$\bfF_G$ be~(a skeleton of) the full subcategory of the category of~$G$--sets on the free~$G$--sets with finitely many orbits: those of the form~$\coprod_r G$.
The functor~\hbox{$\rmF_G\colon\bfE\to\bfF_G$} sending~$r$ to~$\coprod_r G$ is a Lawvere theory, called the theory of~$G$--sets. In particular, for the trivial group~\hbox{$G=\{e\}$}, we have the Lawvere theory~$E$ of sets. 
\end{example}

\begin{remark}
Some authors prefer to work with the opposite category~$\bfF_T^{\op}$, so that the object~$T_r$ is the {\em product}~(rather than the co-product) of~$r$ copies of the object~$T_1$. For example, this was Lawvere's convention when he introduced this notion in~\cite{Lawvere:PNAS}. Our convention reflects the point of view that the object~$T_r$ should be thought of as the free~$T$--model~(or~$T$--algebra) on~$r$ generators, covariantly in~$r$~(or rather in~$\bfE$). To make this precise, recall the definition of a model~(or algebra) for a theory~$T$.
\end{remark}

\begin{definition}\label{def:models}
Given a Lawvere theory~$T$, a {\em$T$--model}~(or \emph{$T$--algebra}) is a presheaf~$X$~(of sets) on the category~$\bfF_T$ that sends~(categorical) sums in~$\bfF_T$ to~(categorical, i.e.~Cartesian) products of sets.~(This means that the canonical map~\hbox{$X(T_r+T_s)\to X(T_r)\times X(T_s)$} induced by the injections is a bijection for all sets~$\underline{r}$ and~$\underline{s}$ in~$\bfE$.) 
We write~$\bfM_T$ for the category of~$T$--models, and we write~$\bfM_T(X,Y)$ to denote the set of morphisms~$X\to Y$ between~$T$--algebras.  Such a morphism is defined to be  a map of presheaves, i.e., a natural transformation, so that~$\bfM_T$ is a full subcategory of the category of presheaves on~$\bfF_T$.
\end{definition}

The values of a~$T$--model are determined up to isomorphism by the value at~$T_1$, and we often use the same notation for a model and its value at~$T_1$.

\begin{example}
The categories of models for the Lawvere theories of Examples~\ref{ex:Amodules} and~\ref{ex:Gsets} are the categories of~$A$--modules and~$G$--sets, respectively. For example, the action of~$G$ on itself from the right gives for each~$g\in G$ a~$G$--map~\hbox{$g\colon \coprod_1G\to\coprod_1G$} in the category~$\bfF_G$ of Example~\ref{ex:Gsets}. Given a model~$X\colon\bfF_G^\op\to\Sets$, the set maps~\hbox{$X(g)\colon X(\coprod_1 G)\to X(\coprod_1 G)$} combine to produce the action of the group~$G$ on the set~$X(\coprod_1 G)$.
\end{example}

\begin{example}\label{ex:free}
The co-variant Yoneda embedding~$\bfF_T\to\Pre(\bfF_T)
$ sends the object~$T_r$ of~$\bfF_T$ to the presheaf~$T_s\mapsto\bfF_T(T_s,T_r)$ represented by it. Such a presheaf is readily checked to be a~$T$--model. We refer to a~$T$--model of this form as~\emph{free}. The definitions unravel to give natural bijections~\hbox{$\bfM_T(T_r,X)\cong X^r$} for~$T$--models~$X$, so that~$T_r$ is indeed a free~$T$--model on~$r$ generators. 
\end{example}

We can summarize the situation as follows. The Yoneda embedding of~$\bfF_T$ into presheaves on~$\bfF_T$ factors~\hbox{$\bfF_T\to\bfM_T\to\Pre(\bfF_T)$} through the category~$\bfM_T$ of~$T$--models. Both functors are fully faithful, and the free~$T$--models are those in the~(essential) image of the first functor.

\begin{definition}\label{def:morphism}
A \emph{morphism}~$S\to T$ between Lawvere theories is a functor~\hbox{$L\colon\bfF_S\to\bfF_T$} that~(strictly) preserves sums. This is equivalent to the condition that~$\rmF_T\cong L\circ\rmF_S$, i.e., that~$L$ is a map under~$\bfE$.
\end{definition}

It is common to describe a morphism~$S\to T$ between two Lawvere theories by giving a functor~\hbox{$R\colon\bfM_T\to\bfM_S$} that is compatible with the forgetful functors to the category~$\bfM_E$ of sets. In this case,~$R$ has a left-adjoint by Freyd's adjoint functor theorem and~$L$ is induced by the restriction of the left adjoint to~$R$ to free models.

For any Lawvere theory~$T$, the category~$\bfM_T$ of~$T$--models is complete and cocomplete. Limits are constructed pointwise, and the existence of colimits follows from the adjoint functor theorem.
The category~$\bfM_T$ becomes symmetric monoidal with respect to the~(categorical) sum, and the unit object~$T_0$ for this structure is also an initial object in the category~$\bfM_T$.


\section{Algebraic~K-theory and stable homology}\label{sec:K}

In this section, we define the algebraic~K-theory spectrum~$\rmK(T)$ of a Lawvere theory~$T$, show how it encodes the stable homology of the automorphism groups of free~$T$--models, and prove our positive results on Morita invariance.

We first specify the constructions of~K-theory we use in this paper.  Our primary approach is to view Lawvere theories as a special case of symmetric monoidal categories and apply the classic constructions of~K-theory for the latter. There are several ways of approaching these constructions; we begin with a brief overview.

Let~$\bfS$ denote a symmetric monoidal groupoid. For the following to make sense,~$\bfS$ needs to satisfy an additional assumption, but we show in Proposition~\ref{prop:stabilization_is_faithful} that this is always the case for the categories we are interested in. We can then pass to Quillen's categorification~$\bfS^{-1}\bfS$ of the Grothendieck construction. The canonical morphism~$\rmB\bfS\to\rmB\bfS^{-1}\bfS$ between the classifying spaces is a group completion, and the target is an infinite loop space. We refer to~\cite{Grayson} and Thomason's particularly brief and enlightening discussion~\cite{Thomason} for detail. To build a~K-theory spectrum~$\rmK(\bfS)$ with underlying infinite loop space~$\Omega^\infty\rmK(\bfS)\simeq\rmB\bfS^{-1}\bfS$, we can use Segal's definition of the algebraic~K-theory of a symmetric monoidal category in terms of~$\Gamma$--spaces. The equivalence comes from~\cite[\S4]{Segal}, where he shows that~$\Omega^\infty\rmK(\bfS)$ is also a group completion of~$\rmB\bfS$. 


\begin{defn}\label{def:K(T)}
Let~$T$ be a Lawvere theory. The {\it algebraic~K-theory} of~$T$ is the spectrum
\begin{equation}
\rmK(T)=\rmK(\bfF_T^\times),
\end{equation}
that is, the spectrum corresponding to the symmetric monoidal groupoid~$\bfF_T^\times$ of isomorphisms in the symmetric monoidal category~$\bfF_T$ of finitely generated free~$T$--models, where the monoidal structure is given by the categorical sum.
\end{defn}

Since the category~$\bfF_T$ can be identified with the symmetric monoidal category of finitely generated free~$T$--models, Definition~\ref{def:K(T)} concerns the algebraic~K-theory of finitely generated free~$T$--models. In particular, the group~$\rmK_0(T)=\pi_0\rmK(T)$ is the Grothendieck group of isomorphism classes of finitely generated free~$T$--models. This group is always cyclic, generated by the isomorphism class~$[\,T_1\,]$ of the free~$T$--model on one generator. However, the group~$\rmK_0(T)$ does not have to be infinite cyclic, as the Examples~\ref{ex:Cantor} and~\ref{ex:K_0M_n} below show.

\begin{remark}\label{rem:surjection}
A morphism~$S\to T$ of Lawvere theories~(as in Definition~\ref{def:morphism}) induces, via the left-adjoint functor~$\bfF_S\to\bfF_T$, a morphism~$\rmK(S)\to\rmK(T)$ of algebraic~K-theory spectra. The left adjoint~$\bfF_S\to\bfF_T$ sends the free~$S$--model~$S_1$ on one generator to the free~$T$--model~$T_1$ on one generator. It follows that the induced homomorphism~$\rmK_0(S)\to\rmK_0(T)$ between cyclic groups is surjective, being the identity on representatives.
\end{remark}

One reason for interest in the algebraic~K-theory of Lawvere theories is the relation to the stable homology of the sequence of automorphism groups attached to a Lawvere theory. We now make this relation made precise.

Let~$T$ be a Lawvere theory. The automorphism groups of the free algebras~$T_r$ often turn out to be very interesting~(see the examples in Section~\ref{sec:examples} below). We use the notation~$\Aut(T_r)$ for these groups. 

Given integers~$r,s\geqslant0$, there is a {\it stabilization} homomorphism
\begin{equation}\label{eq:stabilization}
\Aut(T_r)\longrightarrow\Aut(T_{r+s})
\end{equation}
that `adds' the identity of the object~$T_s$ in the sense of the categorical sum~$+$,  and we use additive notation for this operation. More precisely, stabilization sends an automorphism~$u$ of~$T_r$ to the automorphism of~$T_{r+s}$ that makes the diagram
\[
\xymatrix{
T_{r+s}\ar@{-->}[rr]
&&T_{r+s}\\
T_r+T_s\ar[u]^\cong\ar[rr]^-{u+{T_s}}_\cong
&&T_r+T_s\ar[u]_\cong
}
\]
commute. By abuse of notation, this automorphism of the object~$T_{r+s}$ will sometimes also be denoted by~\hbox{$u+T_s$}. 

\begin{remark}
The alert reader will have noticed that we have not specified our choice of isomorphism~$T_r+T_s\cong T_{r+s}$ in the preceding diagram.  While the requirement that~$T_{r+s}$ be the sum of~$T_r$ and~$T_s$ provides a canonical identification here, we could in fact use any choice of isomorphism. All such choices obviously differ by conjugation by an automorphism of~$T_{r+s}$, so that they induce the same map in homology, which is all that matters for the purposes of this section.
\end{remark}

\begin{proposition}\label{prop:stabilization_is_faithful}
For every Lawvere theory~$T$, the stabilization maps~$\Aut(T_r)\to\Aut(T_{r+1})$ are injective.
\end{proposition}

\begin{proof}
It is enough to show that the kernels are trivial. This is clear for~$r=0$, since~$T_0$ is initial, so that~$\Aut(T_0)$ is the trivial group. For positive~$r$ we can choose a retraction~$\rho$ of the canonical embedding~$\sigma\colon T_r\to T_{r+1}$. If~$u$ is in the kernel of the stabilization map, then we have the following commutative diagram.
\[
\xymatrix{
T_r\ar[d]_\sigma\ar[r]^-u&T_r\ar[d]_\sigma\ar@{=}[dr]&\\
T_{r+1}\ar[r]_-\id&T_{r+1}\ar[r]_-\rho&T_r
}
\]
It implies~$u=\id$.
\end{proof}


Stabilization leads to a diagram
\begin{equation}\label{eq:stableG}
\Aut(T_0)\longrightarrow 
\Aut(T_1)\longrightarrow 
\Aut(T_2)\longrightarrow 
\Aut(T_3)\longrightarrow\cdots
\end{equation}
of groups for every Lawvere theory~$T$. We  write~$\colim_r\Aut(T_r)$ for the colimit of the diagram~\eqref{eq:stableG} with respect to the stabilization maps. This is the {\em stable automorphism group} for the Lawvere theory~$T$.

Let us record the following group theoretical property of the stable automorphism groups. This is presumably well-known already in more or less generality. We nevertheless include an argument here for completeness' sake.

\begin{proposition}\label{prop:commutator_subgroup_perfect}
For every Lawvere theory~$T$, the commutator subgroup of the stable automorphism group~$\colim_r\Aut(T_r)$ is perfect.
\end{proposition}

\begin{proof}
Given a commutator in the group~$\colim_r\Aut(T_r)$, we can represent it as~$[u,v]$ for a pair~$u,v$ of automorphisms in the group~$\Aut(T_r)$ for some~$r$. Allowing us thrice the space, in the group~$\Aut(T_{3r})$ we have the identity
\[
[u,v]+\id(T_{2r})=[u+u^{-1}+\id({T_r}),v+\id({T_r})+v^{-1}].
\]
It therefore suffices to prove that each element of the form~$w+w^{-1}$ is a commutator. This is a version of Whitehead's lemma that holds in every symmetric monoidal category: whenever there are automorphisms~$w_1,\dots,w_n$ of an object such that their composition~$w_1\dotsm w_n$ is the identity, then~$w_1+\dots+w_n$ is a commutator. We apply this to the category~$\bfF_T$ with respect to the monoidal product given by categorical sum~$+$.
\end{proof}

After these preliminaries, we now move on to give another model for the algebraic~K-theory space of a Lawvere theory~$T$, one that uses the Quillen plus construction. This construction led to Quillen's historically first definition of the algebraic~K-theory of a ring~\cite{Quillen:ICM1970}~(see also~\cite{Wagoner} and~\cite{Loday}).

The plus construction can be applied to connected spaces~$X$ for which the fundamental groups have perfect commutator subgroups. It produces a map~\hbox{$X\to X^+$} into another connected space~$X^+$ with the same integral homology, and such that the induced maps on fundamental groups are the abelianization. In fact, these two properties characterize the plus construction. By Proposition~\ref{prop:commutator_subgroup_perfect}, the commutator subgroup of~$\colim_r\Aut(T_r)$ is perfect. Therefore, the plus construction can be applied the classifying space~$\rmB\!\colim_r\Aut(T_r)$ in order to produce another space~$\rmB\!\colim_r\Aut(T_r)^+$. 

\begin{theorem}\label{thm:colimit}
For every Lawvere theory~$T$, there is an equivalence
\begin{equation}\label{eq:product}
\Omega^\infty\rmK(T)\simeq\rmK_0(T)\times\rmB\!\colim_r\Aut(T_r)^+
\end{equation}
of spaces.
\end{theorem}

\begin{proof}
Quillen, in the his proof that the plus construction of~K-theory agrees with the one obtained from the Q-construction, takes an intermediate step~(see~\cite[p.~224]{Grayson}): he shows that the plus construction, together with~$\rmK_0$, gives a space that is equivalent to the classifying space of his categorification~$\bfS^{-1}\bfS$ of the Grothendieck construction of a suitable symmetric monoidal category~$\bfS$. This part of his argument applies here to show that there is an equivalence
\[
\rmK_0(T)\times\rmB\!\colim_r\Aut(T_r)^+\simeq\rmB((\bfF_T^\times)^{-1}\bfF_T^\times
)\]
of spaces for every Lawvere theory~$T$. The claim follows because we already know that the right hand side has the homotopy type of~$\Omega^\infty\rmK(T)$.
\end{proof}

In general, there seems to be no reason to believe that an artificial  product such as the one in~\eqref{eq:product} would form a meaningful whole~(see~\cite[Warning~2.2.9]{Schlichting}). The present case is special because~$\rmK_0(T)$ is generated by the isomorphism class of the free~$T$--algebra~$T_1$ of rank~$1$. Other constructions of the same homotopy type do not separate the group~$\rmK_0(T)$ of components from the rest of the space. One way or another, note that all components of the algebraic~K-theory space~$\rmK(T)$ are equivalent; the group~$\rmK_0(T)$ of components acts transitively on the infinite loop space~$\Omega^\infty\rmK(T)$ up to homotopy.

Since the plus construction does not change homology, the definition of the algebraic~K-theory space immediately gives the following result. 

\begin{theorem}\label{thm:stable_homology}
For every Lawvere theory~$T$, there is an isomorphism
\[
\colim_r\rmH_*(\Aut(T_r))\cong\rmH_*(\Omega^\infty_0\rmK(T))
\]
between the stable homology of the automorphism groups of finitely generated free objects of the theory~$T$ and the homology of the zero component~$\Omega^\infty_0\rmK(T)$ of the algebraic~K-theory space~$\Omega^\infty\rmK(T)$.
\end{theorem}

Ideally, the algebraic~K-theory spectrum~$\rmK(T)$ is more accessible and easier to understand and describe than the stable automorphism group~$\colim_r\Aut(T_r)$. This is not at all plausible from the definition; only the now-classical methods of algebraic~K-theory, which have been developed over half a century, allow us to take this stance.  From this perspective, Theorem~\ref{thm:stable_homology} should be thought of as a computation of the group homology, once the spectrum~$\rmK(T)$ is identified. The examples in Sections~\ref{sec:examples} and~\ref{sec:Boole} give a taste of the flavor of some non-trivial~(and non-linear) cases.


\section{Some non-linear examples}\label{sec:examples}

The goal of this section is to demonstrate the interest in the algebraic~K-theory~$\rmK(T)$ of Lawvere theories~$T$ beyond what are arguably the most fundamental examples, the theories of modules over rings:

\begin{example}\label{ex:modules} 
Consider the theory of modules over a ring~$A$, as in Example~\ref{ex:Amodules}. The automorphism group of the free~$A$--module~$A^r$ of rank~$r$ is the general linear group~\hbox{$\Aut(A^r)=\GL_r(A)$}. The algebraic~K-theory spectrum~$\rmK(A)$ is Quillen's algebraic~K-theory~(actually, the `free' version). In particular~$\rmK(\bbZ)$ is the~K-theory spectrum of the Lawvere theory of abelian groups, in the guise of~$\bbZ$--modules.
\end{example}

We can now move on to discuss non-linear examples: theories that are not given as modules over a ring.

\begin{example}\label{ex:sets}
Consider the initial theory~$E$ of sets. The automorphisms are just the permutations, and the automorphism group~\hbox{$\Aut\{1,\dots,r\}=\Sigma(r)$} is the symmetric group on~$r$ symbols. The algebraic~K-theory is the sphere spectrum:~$\rmK(E)\simeq\bbS$. This is one version of the Barratt--Priddy theorem~\cite{Priddy,Barratt+Priddy}. We go into detail so that we can use the same notation later as well: Let~$\rmQ\simeq\Omega^\infty\bbS$ denote the infinite loop space of stable self-maps of the spheres. The path components of the space~$\rmQ$ are indexed by the degree of the stable maps, as a reflection of~\hbox{$\pi_0(\bbS)=\bbZ$}, and we will write~$\rmQ(r)$ for the component of maps of degree~$r$. There are maps~$\rmB\Sigma(r)\to\rmQ(r)$ which are homology isomorphisms in a range that increases with~$r$ by Nakaoka stability~\cite{Nakaoka}. These maps fit together to induce a homology isomorphism
\begin{equation}\label{eq:SigmaQ}
\rmB\Sigma(\infty)\to\rmQ(\infty)
\end{equation}
between the colimits. The stabilization~$\rmQ(r)\to\rmQ(r+1)$ is always an equivalence, so that all the maps~$\rmQ(r)\to\rmQ(\infty)$ to the colimit are equivalences as well. Passing to group completions, the map~\eqref{eq:SigmaQ} induces an equivalence~\hbox{$\Omega^\infty_0\rmK(E)\simeq\Omega^\infty_0\bbS$} of infinite loop spaces, so that~$\rmK(E)\simeq\bbS$ as spectra. We refer to Morava's notes~\cite{Morava} for more background and for relations to the algebraic~K-theory of the finite fields~$\bbF_q$ when the number~$q$ of elements goes to~$1$.
\end{example}

\begin{example}\label{ex:Segal} 
More generally, for any discrete group~$G$, we can consider the Lawvere theory of~$G$--sets. The algebraic~K-theory spectrum of the Lawvere theory of~$G$--sets is~\hbox{$\rmK(\text{$G$--$\Sets$})\simeq\Sigma^\infty_+(\rmB G)$}, the suspension spectrum of the classifying space~$\rmB G$~(with a disjoint base point~$+$). This observation is attributed to Segal. In particular, for the Lawvere theory~$\bbZ$--sets, this gives
\[
\rmK(\text{$\bbZ$--$\Sets$})\simeq\Sigma^\infty_+(\rmB\bbZ)\simeq\Sigma^\infty_+(\rmS^1)\simeq\bbS\vee\Sigma\bbS.
\]
The theory~$\bbZ$--sets is the theory of permutations~\cite{Szymik:permutations}: a model is a set together with a permutation of that set.
\end{example}

\begin{example}\label{ex:Galatius}
Consider the theory~$\Groups$ of~(all) groups. In this case, the automorphism groups~$\Aut(\rmF_r)$ are the automorphism groups of the free groups~$\rmF_r$ on~$r$ generators. The algebraic~K-theory space has been shown to be the infinite loop space underlying the sphere spectrum by Galatius~\cite{Galatius}: the unit~$\bbS\simeq\rmK(\Sets)\to\rmK(\Groups)$ is an equivalence.
\end{example}

The theory of abelian groups has been dealt with in Example~\ref{ex:modules}.

\begin{example}\label{ex:nil_tower}
There is an interpolation between the theory of all groups and the theory of all abelian groups by the theories~$\Nil_c$ of nilpotent groups of a certain class~$c$, with~\hbox{$1\leqslant c\leqslant\infty$}. There is a corresponding diagram
\[
\xymatrix@R=15pt{
&&\vdots\ar[d]\\
&&\rmK(\Nil_3)\ar[d]\\
&&\rmK(\Nil_2)\ar[d]\\
\bbS\ar@{=}[r]&\rmK(\Groups)\ar[r]\ar[ur]\ar[uur]&\rmK(\Abel)\ar@{=}[r]&\rmK(\bbZ)
}
\]
of algebraic~K-theory spectra. This tower has been studied from the point of view of homological stability and stable homology in~\cite{Szymik:twisted} and~\cite{Szymik:rational}, respectively.
\end{example}

\begin{example}\label{ex:monoids}
In contrast to groups, the algebraic~K-theory of the Lawvere theory~$\Monoids$ of~(associative) monoids~(with unit) is easy to compute: the free monoid on a set~$X$ is modeled on the set of words with letters from that set, and it has a unique basis: the subset of words of length one, which can be identified with~$X$. This implies that the automorphism group of the free monoid on~$r$ generators is isomorphic to the symmetric group~$\Sigma(r)$, so that the map~$\rmK(E)\to\rmK(\Monoids)$ from the algebraic~K-theory of the initial theory~$E$ of sets is an equivalence. By Example~\ref{ex:sets}, we get an equivalence~$\rmK(\Monoids)\simeq\bbS$ of spectra. It follows, again from Galatius's theorem~(see Example~\ref{ex:Galatius}), that the canonical morphism~$\rmK(\Monoids)\to\rmK(\Groups)$ is an equivalence. It would be interesting to see a proof of this fact that does not depend on his result.
\end{example}

\begin{example}\label{ex:Cantor}
Let~$a\geqslant 2$ be an integer. A {\em Cantor algebra} of arity~$a$ is a set~$X$ together with a bijection~$X^a\to X$. The Cantor algebras of arity~$a$ are the models for a Lawvere theory~$\Cantor_a$, and its algebraic~K-theory has been computed in~\cite{Szymik+Wahl}:
\begin{equation}\label{eq:Moore}
\rmK(\Cantor_a)\simeq\bbS/(a-1),
\end{equation}
the mod~\hbox{$a-1$} Moore spectrum . In particular, the spectrum~$\rmK(\Cantor_2)$ is contractible. Note that the definition makes sense for~\hbox{$a=1$} as well. In that case, we have an isomorphism between~$\Cantor_1$ and the Lawvere theory~$\bbZ$--Sets of permutations, and the equivalence~\eqref{eq:Moore} is still true by Example~\ref{ex:Segal}.
\end{example}

\begin{example} 
Lawvere theories can be presented by generators and relations. The `generators' of a theory are specified in terms of a graded set~\hbox{$P=(\,P_a\,|\,a\geqslant0\,)$}, where~$P_a$ is a set of operations of arity~$a$. There is a free Lawvere theory functor~\hbox{$P\mapsto T_P$} that is left adjoint to the functor that assigns to a theory the graded set of operations. For instance, let~$[a]$ be the graded set that only has one element, and where the degree of that element is~$a$. Then~$T_{[a]}$ is the free theory generated by one operation of arity~$a$. For instance, the Lawvere theory~$T_{[0]}$ is the theory of pointed sets. The Lawvere theory~$T_{[1]}$ is the theory of self-maps~(or~$\bbN$--sets): sets together with a self-map, and~$T_{[2]}$ is the theory of magmas: sets equipped with a multiplication that does not have to satisfy any axioms. The free~$T_{[a]}$--model on a set~$X$ is given by the set of all trees of arity~$a$ with leaves colored in~$X$. This model has a unique basis: the trees of height~$1$, and we can argue as in Example~\ref{ex:monoids} that~$\rmK(T_{[a]})\simeq\bbS$.
\end{example}

Finally, we mention the two trivial~(or {\em inconsistent}, in Lawvere's terminology) examples of theories where the free model functor is not faithful~(see Lawvere's thesis~\cite[II.1, Prop.~3]{Lawvere:thesis}).

\begin{example}
There is a theory such that all models are either empty or singletons. It has no operations aside from the projections~$X^n\to X$, and the relations require that all these projections are equal, so that~$x_1=x_2$ for all elements~$x_j$ in a set~$X$ that is a model.
\end{example}

\begin{example}\label{ex:modules_over_the _trivial_ring}
There is a theory such that all models are singletons. It has a~$0$--ary operation~(constant)~$e$, and the relation~$x=e$ has to be satisfied for all~$x$ in a model~$X$. Another way of describing the same Lawvere theory: this is the theory of modules over the trivial ring, where~\hbox{$0=1$}. From this perspective, the theory is not so exotic after all!
\end{example}

For both of these examples, the algebraic~K-theory spectra are obviously contractible.


\section{Morita equivalences and invariance for matrix theories}\label{sec:matrix}

Given a Lawvere theory~$T$ and an integer~$n\geqslant1$, the {\em matrix theory}~$\rmM_n(T)$ is the Lawvere theory such that the free~$\rmM_n(T)$--model on a set~$X$ is the free~$T$--model on the set~\hbox{$\underline{n}\times X$}~(see~\cite[Sec.~4]{Wraith}). 
In other words, the category~$\bfF_{\rmM_n(T)}$ is the full subcategory of the category~$\bfF_{T}$ consisting of the objects~$T_{nr}$ for~$r\geqslant0$.  

More diagrammatically, we may view~$\underline{n}\times -$ as a strong monoidal endofunctor of~$\bfF_T$, which takes an object~$T_r$ to the~$n$--fold sum of~$T_r$ with itself.  The underlying category of the Lawvere theory~$\rmM_n(T)$ is the image of~$\underline{n}\times -$ and the structure functor that defines~$M_n(T)$ as a Lawvere theory is the composite
\[
\bfE\xrightarrow{\phantom{n\times -}}\bfF_T\xrightarrow{\underline{n}\times -}\bfF_T.
\]

It is easy to describe all~$\rmM_n(T)$--models up to isomorphism: given a~$T$--model~$X$, we can construct an~$\rmM_n(T)$--model on the~$n$--th cartesian power~$X^n$ of~$X$; the~$r$--ary~$\rmM_n(T)$--operations~\hbox{$(X^n)^r\to X^n$} are the maps such that all components~\hbox{$(X^n)^r\to X$} are~$nr$--ary~$T$--operations on~$X$. In particular, we get a unary operation~$X^n\to X^n$ for each self-map of the set~$\underline{n}$, and so the monoid~$\End(\underline{n})$ acts on the model~$X^n$.  Every model arises this way, up to isomorphism. Every~$\rmM_n(T)$--model of the form~$X^n$ has an underlying~$T$--model consisting of the operations that are themselves~$n$--th powers, which gives a forgetful functor~$\bfM_{\rmM_n(T)}\to \bfM_T$. Equivalently, there is a morphism
\begin{equation}\label{eq:matrix_morphism}
T\longrightarrow\rmM_n(T)
\end{equation}
of Lawvere theories. From the diagrammatic perspective, this morphism is simply the above functor~$\underline{n}\times-\colon \bfF_T\to \bfF_{\rmM_n(T)}\subset \bfF_T$, which by construction is a functor under~$\bfE$ and hence a map of Lawvere theories.  We readily observe that there are isomorphisms~\hbox{$\rmM_1(T)\cong T$} and~\hbox{$\rmM_m(\rmM_n(T))\cong\rmM_{mn}(T)$}. If~$T$ is the theory of modules over a ring~$A$ as in Example~\ref{ex:Amodules}, then~$\rmM_n(T)$ is the theory of modules over the matrix ring~$\rmM_n(A)$.
The Lawvere theory~$\rmM_n(E)$ is the theory of~$\End(\underline{n})$--sets.

We now show that the higher algebraic~K-theory of a Lawvere theory~$T$ is invariant under passage to matrix theories~$\rmM_n(T)$. 

\begin{theorem}\label{thm:matrix_invariance}
For every Lawvere theory~$T$, there is an equivalence
\[
\Omega^\infty_0\rmK(\rmM_n(T))\simeq\Omega^\infty_0\rmK(T)
\]
of infinite loop spaces.
\end{theorem} 

\begin{proof}
We may use that the existence of isomorphisms~$\rmM_n(T)_r\cong T_{n\times r}$ of models implies that we have isomorphisms
\[
\Aut(\rmM_n(T)_r)\cong\Aut(T_{n\times r})
\]
between the automorphism groups. Therefore, when we compare the diagrams~\eqref{eq:stableG}, the one with the groups~$\Aut(\rmM_n(T)_r)$ for~$\rmM_n(T)$ naturally embeds as a cofinal subdiagram of the diagram with the groups~$\Aut(T_r)$ for~$T$. We only see every~$n$--th term, but the colimits can be identified, of course, and this proves the statement on the level of spaces. 

To see that we have an equivalence of infinite loop spaces, we show that this map is induced by a map of spectra. However, the equivalence is {\it not} induced by the morphism~\hbox{$\rmK(T)\to\rmK(\rmM_n T)$} of spectra that comes from the canonical morphism~\eqref{eq:matrix_morphism} of theories. A remedy is to leave the world of Lawvere theories for the rest of the proof and use the general context of symmetric monoidal theories. Then we see that the equivalence {\it does} come from a morphism~\hbox{$\rmK(\rmM_n T)\to\rmK(T)$} of spectra in the other direction. This morphism of spectra is obtained from the symmetric monoidal functor~$\bfF_{\rmM_n(T)}\to\bfF_T$ given by the inclusion of~$\bfF_{\rmM_n(T)}$ into~$\bfF_T$ as the image of the functor~$\underline{n}\times-$.  This functor is defined by~\hbox{$\rmM_n(T)_r\cong T_{n\times r}\mapsto T_{n\times r}$} and so, while it is essentially the identity on morphisms, it is not necessarily surjective on objects. In particular, it need not be surjective on the level of components, as is required for a map of Lawvere theories according to Remark~\ref{rem:surjection}. On the component of zero, however, it has the effect described in the first part of the proof, showing that we have an equivalence of infinite loop spaces.
\end{proof}

In fact, as tempting as it might be to hope for an equivalence~$\rmK(\rmM_n T)\simeq\rmK(T)$ of~K-theory spectra, we {\it cannot} have that, in general, because of the difference in the groups~$\rmK_0$ of components: 

\begin{example}\label{ex:K_0M_n}
As explained in~\cite[Rem.~5.3]{Szymik+Wahl} and Example~\ref{ex:Cantor} of the following section, the Cantor theories~$\Cantor_a$  of arity~\hbox{$a\geqslant2$} have~$\rmK_0(\Cantor_a)=\bbZ/(a-1)$ finite. But by construction, the matrix theory~$\rmM_n(\Cantor_a)$ only involves the elements represented by multiples of~$n$ in the group~$\bbZ/(a-1)$. Therefore, if~$n$ is not coprime to~$a-1$, then~$\rmK_0(\rmM_n\Cantor_a)$ will be strictly smaller than~$\rmK_0(\Cantor_a)$. In particular, the morphisms between~$\rmK(\Cantor_a)$ and~$\rmK(\rmM_n\Cantor_a)$ described in Remark~\ref{rem:surjection} and Example~\ref{ex:K_0M_n} are {\it not} equivalences in this case.
\end{example}

Theorem~\ref{thm:matrix_invariance} might suggest that the higher algebraic~K-theory of Lawvere theories is Morita invariant, but we show in the rest of the paper that this is not the case. We start with the definition.

\begin{definition}\label{def:Morita_eq}
Two Lawvere theories~$S$ and~$T$ are called {\it Morita equivalent} if their categories~$\bfM_S$ and~$\bfM_T$ of models are equivalent. 
\end{definition}

For instance, if~$S$ is the Lawvere theory of modules over a ring~$A$, then~$T$ is also a Lawvere theory of modules over a ring~$B$, and this ring~$B$ is Morita equivalent to~$A$ in the usual sense; see~\cite[Ex.~3.1]{ASS}. Thus, Definition~\ref{def:Morita_eq} is in agreement with the established terminology for Lawvere theories that are given by rings. In general, it turns out that the Morita equivalence relation is generated by two processes, one of which we have already seen.

\begin{proposition}[{\bf\cite{Dukarm},~\cite{McKenzie}}]
A Lawvere theory is Morita equivalent to a given Lawvere theory~$T$ if and only if it is an idempotent modification of a matrix theory of~$T$ for some pseudo-invertible idempotent of the matrix theory. 
\end{proposition}

Since behavior of algebraic K-theory on passage to matrix theories is already fully described by our results above, we now turn to idempotent modifications.

Let~$T$ be a Lawvere theory with an idempotent endomorphism~\hbox{$u\colon T_1\to T_1$} of the free~$T$--model~$T_1$ on one generator. We write~$u_n\colon T_n\to T_n$ for the~$n$--fold sum, so that~$u_1=u$. An idempotent~$u$ is {\it pseudo-invertible} if, for some fixed~$k$, there are morphisms~$T_1\to T_k$ and~$T_k\to T_1$ such that their composition around~$u_k\colon T_k\to T_k$ is the identity on~$T_1$.

\begin{lemma}\label{lem:properties}
Consider the following properties for a morphism~$f\colon T_r\to T_s$ in~$\bfF_T$ with respect to a fixed idempotent~$u$.
  \begin{enumerate}[label=\normalfont(\arabic*),nosep]
\item~$f=u_sgu_r$ for some~$g\colon T_r\to T_s$
\item~$u_sf=f=fu_r$
\item~$u_sf=fu_r$
  \end{enumerate}
Then~{\upshape(1)}~$\Leftrightarrow$ {\upshape(2)}~$\Rightarrow$ {\upshape(3)}. We have~{\upshape(2)}~$\Leftarrow$ {\upshape(3)} if and only if~$u=\id$.
\end{lemma}

We define~$\bfF_T^u\leqslant\bfF_T$ to be the subcategory~(!) consisting of the morphisms that satisfy condition~(3) in Lemma~\ref{lem:properties} above. Note that~(1) and~(2) do not define a subcategory in general, because the identity morphisms satisfy~(3), but not necessarily~(1) or~(2). However, we can define a new category structure on the subsets of~$\bfF_T(T_r,T_s)$ consisting of morphisms satisfying conditions~(1) and~(2): these subsets are closed under composition, and the~$u_r$'s act as new identities. This gives another category~$\bfF_{uTu}$ and another Lawvere theory, the {\it idempotent modification}~$uTu$ of~$T$ with respect to the idempotent~$u$. There is a functor~\hbox{$\bfF_T^u\to\bfF_{uTu}$} defined by~$f\mapsto uf=ufu=fu$, and we can, in principle, compare the new Lawvere theory~$uTu$ to~$T$ using the zigzag
\[
\bfF_{uTu}\longleftarrow\bfF_T^u\longrightarrow\bfF_T
\]
of functors defined above, all of which are the identities on objects.  These functors then induce a comparison zigzag of K-theory spectra.

However, this zigzag of K-theory spectra is not generally an equivalence. In the following section, we provide examples of Lawvere theories that are Morita equivalent but have different higher algebraic~K-theory. This also shows that~\cite[Ex.~IV.4.13(a)]{Weibel}, which suggests that the inclusion of a symmetric monoidal category into its idempotent completion should always be cofinal, is lacking an additivity assumption.


\section{Theories equivalent to the theory of Boolean algebras}\label{sec:Boole}

In this section, we present new computations: we determine the algebraic~K-theory of the Lawvere theory of Boolean algebras. Our methods allow us to deal more generally with the~Lawvere theories of~$v$--valued Post algebras. Boolean algebras form the case~$v=2$. The Lawvere theories of~$v$--valued Post algebras are all Morita equivalent to each other. In fact, these form the set of all the Lawvere theories that are equivalent to the theory of Boolean algebras. As a consequence of our computations, we show that algebraic~K-theory is not Morita invariant in general.

Boolean algebras and their relationship to set theory and logic are fundamental for mathematics and well-known. Post algebras were introduced by Rosenbloom~\cite{Rosenbloom}. They are named after Post's work~\cite{Post} on non-classical logics with~$v$ truth values. Later references are Wade~\cite{Wade}, Epstein~\cite{Epstein}, as well as the surveys by Serfati~\cite{Serfati} and 
Dwinger~\cite{Dwinger}, to which we refer for defining equations and explicit models of the free algebras. In the following, we will only recall their definition as a Lawvere theory and what is necessary for our purposes.

We write~$\Map(R,S)$ for the set of all maps from a set~$R$ to a set~$S$. As before, we build on the specific finite sets~\hbox{$\underline{r}=\{a\in\bbZ\,|\,1\leqslant a\leqslant r\}$}. For a fixed integer~$v\geqslant 2$, we now consider the category whose objects are the finite sets of the form~$\Map(\underline{r},\underline{v})$, where~$r$ ranges over all integers~$r\geqslant0$, and whose morphisms are all maps between these sets. By construction, this category has finite products, and every object~$\Map(\underline{r},\underline{v})$ is the~\hbox{$r$--th} power of the object~$\Map(\underline{1},\underline{v})=\underline{v}$. Therefore, the opposite category has finite co-products, and every object is a multiple of one object, the one corresponding to the set~$\Map(\underline{1},\underline{v})$. This opposite category defines the Lawvere theory~$\Post_v$ of~{\it$v$--valued Post algebras}. For~$v=2$, Post's~$v$--valued logic specializes to the~$2$--valued Boolean logic, and we have~\hbox{$\Post_2=\Boole$}, the Lawvere theory of Boolean algebras. Using our description above, this is a well-known consequence of Stone duality: the set of subsets of~$\Map(\underline{r},\underline{2})$ is a free Boolean algebra on~$r$ generators, with~$2^{2^r}$ elements in total. 

Dukarm~\cite[Sec.~3]{Dukarm} notes that the Lawvere theories~$\Post_v$ are all Morita equivalent to each other. After all, for any given integer~$v\geqslant 2$, any finite set is a retract of a set of the form~$\Map(\underline{r},\underline{v})$ for~$r\geqslant0$ large enough. There is no need for us to choose such a retraction.~(The situation is comparable to the abstract existence of isomorphisms~\hbox{$\overline{\bbQ}_p\cong\bbC$} of fields between the algebraic closure~$\overline{\bbQ}_p$ of the field~$\bbQ_p$ of~$p$--adic numbers and the field~$\bbC$ of complex numbers, showing that the isomorphism type of~$\overline{\bbQ}_p$ is independent of~$p$.) In any event, it follows from the existence of such retractions that the idempotent completions of the categories of free~$v$--valued Post algebras are equivalent to the category of non-empty finite sets, regardless of~$v$. Since these idempotent completions are independent of the integer~$v$, so is the Morita equivalence class of~$\Post_v$, by the results recalled in Section~\ref{sec:matrix}. The following theorem shows that, in contrast, higher algebraic~K-theory detects the number~$v$ of truth values, and~K-theory is therefore not fully Morita invariant.

In order to state the result, we need the spectrum~$R^\times$ of units of a commutative ring spectrum~$R$~(see~\cite{MQRT}). This spectrum is defined so that its underlying infinite loop space~$\Omega^\infty R^\times$ is the union of the components of~$\Omega^\infty R$ that represent units, i.e., are invertible, in the ring~$\pi_0R$. The inclusion~\hbox{$\Omega^\infty R^\times\to\Omega^\infty R$} then induces an isomorphism on higher homotopy groups. The inclusion is not, however, a morphism of infinite loop spaces. Instead, the delooping~$R^\times$ of~$\Omega^\infty R^\times$ comes from the~$\rmE_\infty$ multiplication of~$R$. We need the units for the localization~$R=\bbS[1/v]$ of the sphere spectrum~$\bbS$ away from~$v$ and its~$0$--truncation, the Eilenberg--Mac Lane spectrum~$R=\rmH\bbZ[1/v]$. The truncation induces a morphism~\hbox{$\bbS[1/v]^\times\to\rmH\bbZ[1/v]^\times$} of spectra of units. There is also a homomorphism~$\bbZ\to\bbZ[1/v]^\times$ of abelian groups that sends the generator~$1$ to the unit~$v$, which induces a map of Eilenberg--Mac Lane spectra.

\begin{theorem}\label{thm:K(Post_v)}
For every integer~$v\geqslant 2$, there is a homotopy pullback square
\[
\xymatrix{
\rmK(\Post_v)\ar[r]\ar[d] & \bbS[1/v]^\times\ar[d]\\
\rmH\bbZ\ar[r]_-v  & \rmH\bbZ[1/v]^\times
}
\]
of spectra. In particular, we have
\[
\rmK_*(\Post_v)\cong\pi_*(\bbS)/v\text{\upshape--power torsion},
\]
where the~$\pi_*(\bbS)$ are the stable homotopy groups of spheres. 
\end{theorem}

We single out the case~$v=2$ for emphasis:

\begin{corollary}\label{cor:K(Boole)}
We have
\[
\rmK_*(\Boole)\cong\pi_*(\bbS)/2\text{\upshape--power torsion}
\]
for the algebraic~K-theory of the Lawvere theory of Boolean algebras.
\end{corollary}

While Boolean algebras form a comparatively well-known algebraic structure, the~$v$--valued Post algebras are certainly non-standard, and it might come as a surprise that we can prove such results without even revealing their defining operations, let alone the axioms that these operations are required to satisfy. However, as we hope the following proof makes clear, the ability to do so is precisely one of the benefits of our categorical methods.

\begin{proof}[Proof of Theorem~\ref{thm:K(Post_v)}]
By definition, the category of free~$v$--valued Post algebras is equivalent to the opposite of the full subcategory of the category of sets spanned by those sets of the form~$\Map(\underline{r},\underline{v})$. Since these have different cardinalities for different values of~$r$, the isomorphism type of the free~$v$--valued Post algebra of rank~$r$ determines the rank~$r$. Passing to group completion, we get~$\rmK_0(\Post_v)\cong\bbZ\cong\pi_0(\bbS)$, as claimed.

For the higher algebraic~K-theory, we turn toward the automorphism groups. If~$X$ is an object in a category~$\bfC$, we have
\[
\Aut_{\bfC^\op}(X)\cong\Aut_{\bfC}(X)^\op\cong\Aut_{\bfC}(X).
\]
Applied to our situation, this shows that the automorphism group of the free~$v$--valued Post algebra of rank~$r$ is isomorphic to the group of permutations of the set~$\Map(\underline{r},\underline{v})$ of cardinality~$v^r$, and therefore to the symmetric group~$\Sigma(v^r)$ acting on a set of~$v^r$ elements.

Stabilization leads us to the colimit of the diagram
\begin{equation}\label{eq:McDuff+Segal}
\Sigma(1)\longrightarrow
\Sigma(v)\longrightarrow
\Sigma(v^2)\longrightarrow
\cdots\longrightarrow
\Sigma(v^r)\longrightarrow
\cdots,
\end{equation}
where the morphisms are given by multiplication with~$v$: a permutation~$\sigma$ of~$\underline{v}^r$ is sent to the permutation
$\sigma\times\id_v$ of~$\underline{v}^r\times\underline{v}=\underline{v}^{r+1}$, which looks just like~$v$ copies of the permutation~$\sigma$ acting on~$v$ disjoint copies of~$\underline{v}^r$. In other words, the permutation~$\sigma\times\id_v$ is a block sum of~$v$ copies of~$\sigma$.

The diagram~\eqref{eq:McDuff+Segal} has been studied before by McDuff--Segal~\cite[Ex.~(iv)]{McDuff--Segal}, and the following identification of its colimit does not come with any claim on originality.

Picking up our notation from Example~\ref{ex:sets}, we have  maps~$\rmB\Sigma(d)\to\rmQ(d)$ that fit together to form a commutative diagram as follows.
\[
\xymatrix{
\rmB\Sigma(1)\ar[r]^-{\times v}\ar[d]&
\rmB\Sigma(v)\ar[r]^-{\times v}\ar[d]&
\rmB\Sigma(v^2)\ar[r]^-{\times v}\ar[d]&\cdots\\
\rmQ(1)\ar[r]_-{\times v}&
\rmQ(v)\ar[r]_-{\times v}&
\rmQ(v^2)\ar[r]_-{\times v}&\cdots
}
\]
This diagram can be used to compute the group completion of the upper colimit, which is the infinite loop space~$\Omega^\infty_0\rmK(\Post_v)$~by Theorem~\ref{thm:colimit}. This time, in contrast to Example~\ref{ex:sets}, the maps in the lower row are not equivalences, but multiplication by~$v$ in the infinite loop space structure on the~$\rmQ(v^r)\simeq\rmQ(\infty)\simeq\rmQ(0)$. In other words, there is a homology isomorphism from the colimit~$\rmB\Sigma(v^\infty)$ to the localization~$\rmQ(0)[1/v]$ away from~$v$. This homology isomorphism gives, after group completion, an equivalence~\hbox{$\Omega^\infty_0\rmK(\Post_v)\simeq\Omega^\infty_0\bbS[1/v]^\times$} of infinite loop spaces. Noting that the stable homotopy groups of the sphere spectrum in positive degrees are finite, and~$A[1/v]=A/(v\text{--power torsion})$ for finite abelian groups~$A$, we obtain the identification of the higher homotopy groups in the statement of the theorem. In other words, we have a morphism~\hbox{$\rmK(\Post_v)\to\bbS[1/v]^\times$} of spectra that induces an isomorphism on stable homotopy groups in positive degrees.  To complete the identification of the spectrum $\rmK(\Post_v)$,  we need to describe what it does on components. However, the description above shows that $1\in\bbZ\cong\pi_0\rmK(\Post_v)$ is sent to $v\in\bbZ[1/v]^\times\cong\pi_0\bbS[1/v]^\times$, and this observation translates immediately into the homotopy pullback diagram in the statement of the theorem.
\end{proof}


We end this section with an observation which indicates that the relationship between the~K-theories of the Lawvere theory~$E$ of sets and of Boolean algebras, or more generally~$v$--valued Post algebras, is not as simple as Theorem~\ref{thm:K(Post_v)} might suggest.

\begin{proposition}\label{prop:Morava}
For each prime~$p$, the homomorphism
\[
\pi_n(\bbS)\cong\rmK_n(E)\longrightarrow\rmK_n(\Post_p)\cong\pi_n(\bbS)/p\text{--power torsion},
\]
induced by the universal arrow~$E\to\Post_v$ of Lawvere theories, is not surjective. In particular, it is not the canonical surjection.
\end{proposition}

\begin{proof}
Every Boolean algebra has a natural structure of an~$\bbF_2$--vector space. The addition is given by the symmetric difference~\hbox{$x+y=(x\vee y)\wedge\neg(x\wedge y)=(x\wedge\neg y)\vee(\neg x\wedge y).$} In fact, the category of Boolean algebras is isomorphic to the category of Boolean rings, which are commutative rings where {\it every} element is idempotent. If~$2$ is idempotent, we have~\hbox{$4=2^2=2$}, so that~$2=0$, and the underlying abelian group is~$2$--torsion. More generally, if~$p$ is a prime number, every~$p$--Post algebra admits a natural structure of an~$\bbF_p$--algebra in which every element~$x$ satisfies~$x^p=x$~(see~\cite{Wade} or~\cite{Serfati}). 

It follows that the canonical morphism~$\bbS\simeq\rmK(E)\to\rmK(\Post_p)$ of algebraic~K-theory spectra factors through the algebraic~K-theory~$\rmK(\bbF_p)$ of the field~$\bbF_p$.
\[
\bbS\simeq\rmK(E)
\longrightarrow\rmK(\bbF_p)
\longrightarrow\rmK(\Post_p)
\]
On the level of automorphism groups, these morphisms correspond to embeddings
\[
\Sigma(r)\longrightarrow\GL_r(\bbF_p)\longrightarrow\Sigma(p^r)
\]
of groups with images given by the subgroups of~$\bbF_p$--linear bijections and the subgroup of that given by the permutation matrices.

Quillen~\cite[Thm.~8(i)]{Quillen:Annals} has shown that~$\rmK_{2j-1}(\bbF_q)\cong\bbZ/(q^j-1)$ for all~$j\geqslant 1$ and for all prime powers~$q$. It follows that the~$p$--torsion of the higher algebraic~K-groups~$\rmK_n(\bbF_p)$ of~$\bbF_p$ is trivial. On the other hand, his computations~\cite{Quillen:letter} showed that most of the stable homotopy of the spheres is contained in the kernel of the canonical morphisms~\hbox{$\bbS\to\rmK(\bbZ)\to\rmK(\bbF_p)$} of spectra: what is detected in the algebraic~K-theory of finite fields is essentially the image of Whitehead's~J-homomorphism. In particular, the kernel contains much more than just the~$p$--power torsion.
\end{proof}

\begin{remark}
Morava, in his 2008 Vanderbilt talk~\cite{Morava}, highlighted ``the apparent fact that the spectrum defined by the symmetric monoidal category of finite pointed sets under Cartesian product has not been systematically studied.'' This spectrum can be modeled as the algebraic~K-theory of a {\it many-sorted} Lawvere theory, where the sorts correspond to the prime numbers. It is not worth the effort to develop our theory in more generality just to cover that one example. Instead, we have contented ourselves with demonstrating how the theory we have developed so far suffices for us to deal with the local factors corresponding to each prime. 
\end{remark}


\section*{Acknowledgments}

The authors would like to thank the Isaac Newton Institute for Mathematical Sciences, Cambridge, for support and hospitality during the program `Homotopy harnessing higher structures' where work on this paper was undertaken. This work was supported by~EPSRC grant no~EP/K032208/1. The first author was partially supported by the United States National Science Foundation under DMS Grants No.~1710534 and 2104300. The authors also thank the anonymous referees for their insightful feedback, which has improved the paper.



\vfill

Department of Mathematics, Vanderbilt University, 1326 Stevenson Center, Nashville, TN, USA\\
\href{mailto:am.bohmann@vanderbilt.edu}{am.bohmann@vanderbilt.edu}

Department of Mathematical Sciences, NTNU Norwegian University of Science and Technology, 7491 Trondheim, NORWAY\\
\href{mailto:markus.szymik@ntnu.no}{markus.szymik@ntnu.no}

School of Mathematics and Statistics, The University of Sheffield, Sheffield S3 7RH, UNITED KINGDOM\\
\href{mailto:m.szymik@sheffield.ac.uk}{m.szymik@sheffield.ac.uk}


\begin{thebibliography}{R-WW17}

\bibitem[ARV11]{ARV} J. Ad\'amek, J. Rosick\'y, E.M. Vitale. Algebraic theories. A categorical introduction to general algebra. Cambridge Tracts in Mathematics 184. Cambridge University Press, Cambridge, 2011.

\bibitem[ASS06]{ASS} J. Ad\'amek, M. Sobral, L. Sousa. Morita equivalence of many-sorted algebraic theories. J. Algebra 297 (2006) 361--371. 

\bibitem[BP72]{Barratt+Priddy} M. Barratt, S. Priddy. On the homology of non-connected monoids and their associated groups. Comment. Math. Helv. 47 (1972) 1--14.

\bibitem[BH16]{Barrett--Halvorson} T.W. Barrett, H. Halvorson. Morita equivalence. Rev. Symb. Log. 9 (2016) 556--582.

\bibitem[BS]{Bohmann+Szymik:assembly} A.M. Bohmann, M. Szymik. Generalizations of Loday's assembly maps for Lawvere's algebraic theories. J. Inst. Math. Jussieu (to appear).\\
\href{http://arxiv.org/abs/arXiv:2112.07003}{arXiv:2112.07003}

\bibitem[Bor94]{Borceux} F. Borceux. Handbook of categorical algebra. 2. Categories and structures. Encyclopedia of Mathematics and its Applications 51. Cambridge University Press, Cambridge, 1994.

\bibitem[Car79]{Cartier} P. Cartier. Logique, cat\'egories et faisceaux [d'apr\`es F. Lawvere et M. Tierney].  S\'eminaire Bourbaki, 30e ann\'ee (1977/78), Exp. No. 513, 123--146. 
Lecture Notes in Math., 710. Springer, Berlin, 1979. 
  
\bibitem[Duk88]{Dukarm} J.J. Dukarm. Morita equivalence of algebraic theories. Colloq. Math. 55 (1988) 11--17.

\bibitem[Dwi77]{Dwinger} P. Dwinger. A survey of the theory of Post algebras and their generalizations. Modern uses of multiple-valued logic (Fifth Internat. Sympos., Indiana Univ., Bloomington, Ind., 1975) 51--75. Reidel, Dordrecht, 1977.

\bibitem[Eps60]{Epstein} G. Epstein. The lattice theory of Post algebras. Trans. Amer. Math. Soc. 95 (1960) 300--317.

\bibitem[Gal11]{Galatius} S. Galatius. Stable homology of automorphism groups of free groups. Ann. of Math. 173 (2011) 705--768.

\bibitem[Gra76]{Grayson} D. Grayson. Higher algebraic~K-theory II (after Daniel Quillen). Algebraic~K-theory (Proc. Conf., Northwestern Univ., Evanston, Ill., 1976) 217--240. Lecture Notes in Math. 551. Springer, Berlin, 1976.

\bibitem[Law63]{Lawvere:PNAS} F.W. Lawvere. Functorial semantics of algebraic theories. Proc. Nat. Acad. Sci. U.S.A. 50 (1963) 869--872.

\bibitem[Law69]{Lawvere:Dialectica} F.W. Lawvere. Adjointness in Foundations. Dialectica 23 (1969) 281--296.

\bibitem[Law75]{Lawvere:Introduction} F.W. Lawvere. Introduction. Model theory and topoi, 3--14. Lecture Notes in Math. 445. Springer, Berlin, 1975.

\bibitem[Law04]{Lawvere:thesis} F.W. Lawvere. Functorial semantics of algebraic theories and some algebraic problems in the context of functorial semantics of algebraic theories. Repr. Theory Appl. Categ. 5 (2004) 1--121.

\bibitem[Lod76]{Loday} J.-L. Loday.~K-th\'eorie alg\'ebrique et repr\'esentations de groupes. Ann. Sci. \'Ecole Norm. Sup. 9 (1976) 309--377.

\bibitem[May77]{MQRT} J.P. May~$\rmE_\infty$ ring spaces and~$\rmE_\infty$ ring spectra. With contributions by F.~Quinn, N.~Ray, and J.~Tornehave. Lecture Notes in Mathematics 577. Springer-Verlag, Berlin-New York, 1977.

\bibitem[McDS76]{McDuff--Segal} D. McDuff, G. Segal. Homology fibrations and the ``group-completion'' theorem. Invent. Math. 31 (1976) 279--284. 

\bibitem[McK96]{McKenzie} R. McKenzie. An algebraic version of categorical equivalence for varieties and more general algebraic categories. Logic and algebra (Pontignano, 1994) 211--243. Lecture Notes in Pure and Appl. Math. 180. Dekker, New York, 1996. 

\bibitem[Mor]{Morava} J. Morava. Some background for Manin's theorem~$\rmK(\bbF_1)\simeq\bbS$. Talk at the Vanderbilt conference on Noncommutative Geometry and Geometry over the Field with One Element, 2008.\\ 
\href{http://www.alainconnes.org/docs/Morava.pdf}{www.alainconnes.org/docs/Morava.pdf}

\bibitem[Nak60]{Nakaoka} M. Nakaoka. Decomposition theorem for homology groups of symmetric groups. Ann. of Math. 71 (1960) 16--42.

\bibitem[Par69]{Pareigis} B. Pareigis.~Kategorien und Funktoren. Mathematische Leit\-f\"a\-den. B.G. Teubner, Stuttgart, 1969.

\bibitem[Pos21]{Post} E.L. Post. Introduction to a general theory of propositions. Amer. J. Math. 41 (1921) 165--185.

\bibitem[Pri71]{Priddy} S.B. Priddy. On~$\Omega^\infty S^\infty$ and the infinite symmetric group. Algebraic topology (Proc. Sympos. Pure Math., Vol. XXII, Univ. Wisconsin, Madison, Wis., 1970) 217--220. Amer. Math. Soc., Providence, R.I., 1971.

\bibitem[Qui71]{Quillen:ICM1970}  D. Quillen. Cohomology of groups. Actes du Congr\`es International des Math\'e\-ma\-ticiens (Nice, 1970), Tome 2, 47--51. Gauthier-Villars, Paris, 1971.

\bibitem[Qui72]{Quillen:Annals} D. Quillen. On the cohomology and~K-theory of the general linear groups over a finite field. Ann. Math. 96 (1972) 552--586.

\bibitem[Qui76]{Quillen:letter} D. Quillen. Letter from Quillen to Milnor on~$\mathrm{Im}(\pi_i\mathrm{O}\to{\pi_i}^\mathrm{s}\to\mathrm{K}_i\bbZ)$. Algebraic~K-theory (Proc. Conf., Northwestern Univ., Evanston, Ill., 1976) 182--188. Lecture Notes in Math. 551. Springer, Berlin, 1976. 

\bibitem[R-WW17]{R-WW} O. Randal-Williams, N. Wahl. Homological stability for automorphism groups. Adv. Math. 318 (2017) 534--626.

\bibitem[Ros42]{Rosenbloom} P.C. Rosenbloom. Post algebras. I. Postulates and general theory. Amer. J. Math. 64 (1942) 167--188. 

\bibitem[Sch11]{Schlichting} M. Schlichting. Higher algebraic K-theory. Topics in algebraic and topological K-theory 167--241. Lecture Notes in Math. 2008. Springer, Berlin, 2011. 

\bibitem[Sch70]{Schubert} H. Schubert.~Kategorien. II. Heidelberger Taschenb\"ucher 66. Springer-Verlag, 1970.

\bibitem[Seg74]{Segal} G. Segal. Categories and cohomology theories. Topology 13 (1974) 293--312.

\bibitem[Ser73]{Serfati} M. Serfati. Introduction aux alg\`ebres de Post et \`a leurs applications. Cahiers du Bureau Universitaire de Recherche Op\'erationnelle. S\'erie Recherche 21 (1973) 3--100.
  
\bibitem[Szy14]{Szymik:twisted} M. Szymik. Twisted homological stability for extensions and automorphism groups of free nilpotent groups. J.~K-Theory 14 (2014) 185--201.

\bibitem[Szy18]{Szymik:permutations} M. Szymik. Permutations, power operations, and the center of the category of racks. Comm. Algebra 46 (2018) 230--240. 

\bibitem[Szy19]{Szymik:rational} M. Szymik. The rational stable homology of mapping class groups of universal nil-manifolds. Ann. Inst. Fourier 69 (2019) 783--803.

\bibitem[SW19]{Szymik+Wahl} M. Szymik, N. Wahl. The homology of the Higman--Thompson groups. Invent. Math. 216 (2019) 445--518.

\bibitem[Tho80]{Thomason} R.W. Thomason. Beware the phony multiplication on Quillen's~$\mathcal{A}^{-1}\mathcal{A}$. Proc. Amer. Math. Soc. 80 (1980) 569--573.

\bibitem[Tse17]{Tsementzis} D. Tsementzis. A syntactic characterization of Morita equivalence. J. Symb. Log. 82 (2017) 1181--1198.

\bibitem[Wad45]{Wade} L.I. Wade. Post algebras and rings. Duke Math. J. 12 (1945) 389--395. 

\bibitem[Wag72]{Wagoner} J.B. Wagoner. Delooping classifying spaces in algebraic~K-theory. Topology 11 (1972) 349--370. 

\bibitem[Wei13]{Weibel} C.A. Weibel. The K-book. An introduction to algebraic K-theory. Graduate Studies in Mathematics, 145. American Mathematical Society, Providence, RI, 2013. 

\bibitem[Wra71]{Wraith} G.C. Wraith. Algebras over theories. Colloq. Math. 23 (1971) 181--190.

\end{thebibliography}
\end{document}